\documentclass[conference]{IEEEtran}

\usepackage{amsmath, amsfonts, amsthm}
\usepackage{tikz}
\usetikzlibrary{arrows,matrix,decorations.pathmorphing,
  decorations.markings, calc, backgrounds}
\usepackage{mathpartir}
\usepackage[citecolor=blue,linkcolor=purple,urlcolor=blue,colorlinks]{hyperref}
\usepackage[nameinlink,capitalize,noabbrev]{cleveref}
\usepackage{quiver}
\usepackage{xspace}
\usepackage{xcolor}
\usepackage{bbm}
\usepackage{breakurl}
\hypersetup{pdfborder={0 0 100}}
\usepackage{lipsum}
\usepackage{enumitem}
\usepackage{mathtools}
\usepackage[switch]{lineno}

\newtheorem{theorem}{Theorem}
\newtheorem{lemma}[theorem]{Lemma}

\newtheorem{proposition}[theorem]{Proposition}
\newtheorem{corollary}[theorem]{Corollary}
\theoremstyle{definition}
\newtheorem{definition}[theorem]{Definition}
\newtheorem{remark}[theorem]{Remark}
\newtheorem{example}[theorem]{Example}

\newcommand{\teletype}[1]{\ensuremath{\mathtt{#1}}}
\newcommand{\systemname}[1]{\teletype{\color{darkgray}#1}\xspace}

\newcommand{\CubicalAgda}{\systemname{Cubical} \systemname{Agda}}

\definecolor{Revolutionary}{RGB}{232,70,68}

\usepackage{agda}
\newcommand{\anum}[1]{\AgdaNumber{#1}}

\newcommand{\func}[1]{\AgdaFunction{#1}}

\newcommand{\con}[1]{{\AgdaInductiveConstructor{\ensuremath{\mathsf{#1}}}}}
\usepackage{catchfilebetweentags}
\usepackage{newunicodechar}
\newunicodechar{≃}{\ensuremath{\simeq}}
\newunicodechar{≅}{\ensuremath{\cong}}
\newunicodechar{∈}{\ensuremath{\in}}
\newunicodechar{⁻}{\ensuremath{^{-}}}
\newunicodechar{ᴹ}{\ensuremath{_{M}}}
\newunicodechar{ᴺ}{\ensuremath{_{N}}}
\newunicodechar{≡}{\ensuremath{\equiv}}
\newunicodechar{λ}{\ensuremath{\lambda}}
\newunicodechar{⊎}{\ensuremath{\uplus}}
\newunicodechar{∷}{\ensuremath{::}}
\newunicodechar{ℓ}{\ensuremath{\ell}}
\newunicodechar{ᵢ}{\ensuremath{_i}}
\newunicodechar{⟨}{\ensuremath{\langle}}
\newunicodechar{⟩}{\ensuremath{\rangle}}
\newunicodechar{α}{\ensuremath{\alpha}}
\newunicodechar{β}{\ensuremath{\beta}}
\newunicodechar{θ}{\ensuremath{\theta}}
\newunicodechar{φ}{\ensuremath{\varphi}}
\newunicodechar{ψ}{\ensuremath{\psi}}
\newunicodechar{η}{\ensuremath{\eta}}
\newunicodechar{ε}{\ensuremath{\varepsilon}}
\newunicodechar{ι}{\ensuremath{\iota}}
\newunicodechar{Σ}{\ensuremath{\Sigma}}
\newunicodechar{σ}{\ensuremath{\sigma}}
\newunicodechar{∀}{\ensuremath{\forall}}
\newunicodechar{ℕ}{\ensuremath{\mathbb{N}}}
\newunicodechar{→}{\ensuremath{\to}}
\newunicodechar{⊎}{\ensuremath{\uplus}}
\newunicodechar{⋆}{\ensuremath{*}}
\newunicodechar{¬}{\ensuremath{\lnot}}
\newunicodechar{Θ}{\ensuremath{\theta}}
\newunicodechar{ᴸ}{\ensuremath{{}^{\color{black}\leftarrow}}}  
\newunicodechar{ᴿ}{\ensuremath{{}^{\color{black}\rightarrow}}} 
\newunicodechar{∧}{\ensuremath{\wedge}}
\newunicodechar{∨}{\ensuremath{\vee}}
\newunicodechar{∼}{\ensuremath{\sim}}
\newunicodechar{≢}{\ensuremath{\nequiv}}
\newunicodechar{Π}{\ensuremath{\Pi}}
\newunicodechar{ℤ}{\ensuremath{\mathbb{Z}}}
\newunicodechar{∥}{\ensuremath{\parallel}}
\newunicodechar{∣}{\ensuremath{\mid}}
\newunicodechar{ℚ}{\ensuremath{\mathbb{Q}}}
\newunicodechar{₊}{\ensuremath{_+}}
\newunicodechar{∗}{\ensuremath{\ast}}
\newunicodechar{₁}{\ensuremath{_1}}
\newunicodechar{₂}{\ensuremath{_2}}
\newunicodechar{⊥}{\ensuremath{\bot}}
\newunicodechar{∘}{\ensuremath{\circ}}
\newunicodechar{ρ}{\ensuremath{\rho}}
\newunicodechar{↦}{\ensuremath{\mapsto}}
\newunicodechar{₀}{\ensuremath{_0}}
\newunicodechar{∙}{\ensuremath{\boldsymbol{\cdot}}}
\newunicodechar{Ω}{\ensuremath{\Omega}}
\newunicodechar{§}{\ensuremath{\mathbb{S}}}
\newunicodechar{𝟙}{\ensuremath{\mathbb{1}}}
\newunicodechar{×}{\ensuremath{\times}}
\newunicodechar{♥}{\textnormal{\func{Susp}}}
\newunicodechar{₋}{\func{${}_{-}\!$}}
\newunicodechar{✪}{\func{$\sigma\!$}}
\newunicodechar{♢}{$\sphere{2}$}
\newunicodechar{⋆}{\func{\text{\textasteriskcentered}}}
\newunicodechar{®}{${}^{\text{\color{black}{$3$}}}$}
\newunicodechar{ℏ}{,}
\newunicodechar{⌣}{$\smile$}
\newunicodechar{⋈}{\text{\textasteriskcentered}}
\newunicodechar{▬}{${}^2$}
\newunicodechar{𝕒}{\color{black}{\textnormal{$A$}}}
\newunicodechar{𝕓}{\color{black}{\textnormal{$B$}}}
\newunicodechar{𝕔}{$\star_A$}
\newunicodechar{𝕕}{$\star_B$}
\newunicodechar{∇}{$\nabla$}
\newunicodechar{⊓}{\color{black}{\textnormal{$A$}}}
\newunicodechar{𝕝}{\color{black}{\textnormal{$B$}}}
\newunicodechar{𝕗}{$i^{\vee}$}
\newunicodechar{𝟙}{$\mathbbm{1}$}
\newunicodechar{⊗}{$\mathbb{S}^{{\color{black}2}}$}
\newunicodechar{⊕}{$\mathbb{S}^{{\color{black}3}}$}
\newunicodechar{₃}{${}_3$}
\newunicodechar{∃}{$\exists$}
\newunicodechar{≤}{$\leq$}

\newunicodechar{□}{\con{\ensuremath{\square}}}
\newunicodechar{α}{\func{\ensuremath{{\mathbb{T}^2}}}} 

\newunicodechar{β}{\func{\ensuremath{{\mathbb{K}^2}}}} 

\newunicodechar{γ}{\func{\ensuremath{{\mathbb{R}P^2}}}} 
\newunicodechar{◃}{}
\newunicodechar{▹}{}
\newunicodechar{◁}{$\Sigma$}
\newunicodechar{△}{$\,\func{\textsf{Susp}}$}
\newunicodechar{ℜ}{\con{\textsf{inr}}}
\newunicodechar{𝕃}{\con{\textsf{inl}}}
\newunicodechar{⋁}{\func{$\vee$}}
\newunicodechar{✶}{$\star$}

\newunicodechar{ₕ}{\func{\ensuremath{_h}}}
\newunicodechar{ϕ}{\func{\ensuremath{\phi}}}
\newunicodechar{ζ}{\anum{3}\;\;\;\;\;\;\;\;\;}

\newcommand{\ptrunc}[1]{\lVert{#1}\rVert}

\newcommand{\Square}[5]{
\newcommand{\maxVal}{{(29 * (#1)) / 30}}
\newcommand{\minVal}{{(#1) / 30}}
\begin{center}
    \begin{tikzpicture}
    \draw[->] (0 , \minVal) -- (0 , \maxVal);
    \node[anchor = east] at (0 , #1 / 2) {#2} ;

    \draw[->] (\minVal , #1) -- (\maxVal , #1) ;
    \node[anchor = south] at (#1 / 2 , #1) {#3} ;

    \draw[->] (#1 , \minVal) -- (#1 , \maxVal) ;
    \node[anchor = west] at (#1 , #1 / 2) {#4} ;

    \draw[->] (\minVal , 0) -- (\maxVal , 0) ;
    \node[anchor = north] at (#1 / 2 , 0) {#5} ;

    \end{tikzpicture}
\end{center}
\let\maxVal\undefined
\let\minVal\undefined
}

\newlength{\LETTERheight}
\AtBeginDocument{\settoheight{\LETTERheight}{I}}

\newcommand{\bZ}{\ensuremath{\mathbb{Z}}}
\newcommand{\bN}{\ensuremath{\mathbb{N}}}
\newcommand{\sphere}[1]{\ensuremath{\mathbb{S}^{#1}}}

\newcommand{\RP}{\ensuremath{\mathbb{R}P}}
\newcommand{\CP}{\ensuremath{\mathbb{C}P^2}}
\newcommand{\refl}{\textsf{refl}}
\newcommand{\inr}{\textnormal{\textsf{inr}}}
\newcommand{\inl}{\textnormal{\textsf{inl}}}

\newcommand{\trunc}[1]{\ensuremath{\left| #1\right|}}
\newcommand{\truncT}[2]{{\ensuremath \lVert #2 \rVert}_{#1}}

\newcommand{\fib}{\ensuremath{\mathsf{fib}}}

\newcommand{\Susp}[1]{\ensuremath{\Sigma{#1}}}

\definecolor{dkblue}{rgb}{0,0.1,0.5}
\definecolor{lightblue}{rgb}{0,0.5,0.5}
\definecolor{dkgreen}{rgb}{0,0.6,0}
\definecolor{dkbrown}{rgb}{0.4,0,0}
\definecolor{dkviolet}{rgb}{0.6,0,0.8}

\newcommand{\mycomment}[3]{}
\newcommand{\todo}[1]{\mycomment{red}{TODO: }{#1}}

\newcommand{\sq}[1]{\textnormal{\ensuremath{\mathsf{Sq}^{#1}}}}
\newcommand{\sqind}[2]{\textnormal{\ensuremath{\mathsf{Sq}_{#2}^{#1}}}}

\newcommand{\totSq}{\textnormal{\ensuremath{\widehat{\mathsf{Sq}}}}}

\DeclareMathOperator*{\bigast}{\raisebox{-0.6ex}{\scalebox{2.5}{$\ast$}}}
\DeclareMathOperator*{\bigsmile}{\raisebox{-0.ex}{\scalebox{1.5}{$\smile$}}}

\newcommand{\UU}{\mathcal{U}}
\newcommand{\two}{\mathbbm{2}}
\newcommand{\bZtwo}{\bZ/2\bZ}
\newcommand{\pt}{\textup{pt}}

\newcommand{\RPinf}{\textnormal{\ensuremath{\mathbb{R}P^\infty}}}

\newcommand{\bOne}{\ensuremath{\mathbbm{1}}}

\newcommand{\pairfun}[2]{\textnormal{\ensuremath{{\mathsf{Elim}^{#1}_{#2}}}}}

\newcommand{\fst}{\textnormal{\textsf{fst}}}
\newcommand{\snd}{\textnormal{\textsf{snd}}}

\newcommand{\gysin}{\textnormal{\textsf{Gys}}}

\makeatletter
\newcommand{\myitem}[1]{%
\item[#1]\protected@edef\@currentlabel{#1}%
}
\makeatother

\newcommand{\deppath}[3]{
\newlength{\myl}
\settowidth{\myl}{\text{\ensuremath{#3}}}
\begin{tikzcd}[ampersand replacement=\& , column sep = \the\myl]
{#1} \& {#2}
\arrow["{\tiny{#3}}", from=1-1, to=1-2 , squiggly]
\end{tikzcd}
}

\begin{document}

\title{The Steenrod squares via unordered joins}

\author{\IEEEauthorblockN{Axel Ljungström}
 \IEEEauthorblockA{Department of Mathematics\\
   Stockholm University\\
   Email: axel.ljungstrom@math.su.se}
 \and
\IEEEauthorblockN{David Wärn}
 \IEEEauthorblockA{Department of Computer Science and Engineering\\
   University of Gothenburg and Chalmers University of Technology\\
   Email: warnd@chalmers.se}}

\IEEEoverridecommandlockouts

\maketitle
\IEEEpeerreviewmaketitle
\pagestyle{plain}

\begin{abstract} 

The Steenrod squares are cohomology operations with important applications in
algebraic topology. While these operations are well-understood classically,
little is known about them in the setting of homotopy type theory.  Although a
definition of the Steenrod squares was put forward by Brunerie (2017), proofs
of their characterising properties have remained elusive. In this paper, we
revisit Brunerie's definition and provide proofs of these properties, including
stability, Cartan's formula and the Adem relations. This is done by studying a
higher inductive type called the unordered join. This approach is inherently
synthetic and, consequently, many of our proofs differ significantly from their
classical counterparts. Along the way, we discuss upshots and limitations of
homotopy type theory as a synthetic language for homotopy theory. The paper is
accompanied by a computer formalisation in Cubical Agda.

\end{abstract}

\section{Introduction}
\label{sec:intro}

Homotopy type theory (HoTT) is an extension of Martin-Löf type theory based on
the idea of treating \emph{types} as $\infty$-groupoids, or
\emph{spaces}. While HoTT only gained attention as recently as 2012,
$\infty$-groupoids themselves are important mathematical objects and have
long been fruitfully studied using the many tools of algebraic topology and
homotopy theory. A key question is to what extent these tools can be made
to work with the language of HoTT, and whether HoTT can provide new
insights going beyond classical homotopy theory. By now, there is an
established line of research, dubbed \emph{synthetic homotopy theory}, dedicated
to answering these questions.  The promises of synthetic homotopy theory
include conceptual clarity, semantic generality (an argument expressed in
HoTT automatically applies in many models, including arbitrary
$\infty$-topoi~\cite{Shulman19}), and amenability to computer formalisation, but it
also comes with its own set of limitations.

A fundamental tool in homotopy theory is that of \emph{cohomology}, 
and a fundamental tool in making sense of cohomology is that of 
\emph{cohomology operations}. These are ways of constructing new cohomology classes
from old ones, and they give the cohomology of any space a rich structure.
The purpose of this paper is to study an important family of cohomology operations,
the \emph{Steenrod squares}, in HoTT.
Although a good deal of work has been done setting up the foundations of cohomology
in synthetic homotopy theory, from Eilenberg--MacLane spaces~\cite{LicataFinster14},
cohomology groups~\cite{BLM22}, cup products \cite{LLM23,LM24}, and cellular 
cohomology~\cite{BuchholtzFavonia18}, to Gysin sequences~\cite{Brunerie16}
and spectral sequences~\cite{FlorisPhd}, the Steenrod squares have so far only been
\emph{defined} in HoTT in a short text by Brunerie~\cite{Brunerie17}.
The Steenrod squares are classically known to satisfy a list of properties
that are not easily read off from their definition but are important
for applications, and these properties have remained elusive in synthetic homotopy
theory. In this work, we will prove all these properties.

We have computer
\href{https://github.com/caripoulet974/cubical/blob/master/Cubical/Papers/Steenrod.agda}{formalised}
a large part of the project (including all key technical results) in
\CubicalAgda, a proof assistant implementing a flavour of HoTT called
cubical type theory. We emphasise, however, that this paper is
agnostic with respect to HoTT flavour and is written in the
implementation-agnostic informal style of the HoTT book~\cite{HoTT13}.

Two themes feature prominently in this work. The first theme is that of
constructions which depend on an arbitrary 2-element type~\cite{Brunerie17, Pairs}.
Here, a type is said to be a 2-element type if it \emph{merely} is equivalent
to the standard 2-element type $\two$ (i.e. $\{0,1\}$).
We denote the type of all 2-element types by $\RPinf$.
By univalence, any construction that depends on an arbitrary 2-element type (i.e.\ indexed by \RPinf) automatically
respects automorphisms of 2-element types, and in this way we get a synthetic approach to 
what is classically known as equivariant homotopy theory. The second theme is that of higher inductive types;
of particular importance to us are joins and smash products.
The meat of our work consists of studying the interaction of these two themes:
unordered joins and smash products and their properties.

To set the stage and state the main result of this paper, let us briefly
revisit the work of Brunerie. Brunerie defines the $n$th Steenrod square, a map
$H^m(X,\bZtwo) \to H^{m+n}(X,\bZtwo)$, directly on Eilenberg--MacLane spaces:
in HoTT, elements of $H^m(X,\bZtwo)$
are represented by functions $X \to K_m$ where $K_m\coloneqq K(\bZtwo,m)$
denotes the $m$th Eilenberg--MacLane space of $\bZtwo$, 
and so the Steenrod square is given by a map $\sq{n} : K_m \to K_{m+n}$.
This map is defined as a composition:
\begin{align}\label{eq:sqdef}
  K_m \to (\RPinf \to K_{2m}) \xrightarrow{\sim} \Pi_{i \leq 2m} K_i \xrightarrow{\mathsf{proj}_{m+n}} K_{m+n}
  .
\end{align}
Here, the first map is defined using \emph{unordered smash products}, and the second equivalence comes from 
the Thom isomorphism theorem~\cite[Section 6.1]{Brunerie16}.
While this construction is elegant, it has turned out to be difficult to analyse. 
In particular, one would like to know that the Steenrod squares satisfy the properties listed below;
proving these is the primary contribution of this paper.
\begin{theorem}[The Steenrod squares, axiomatically]
  \label{thm:main}
  There is a set of pointed maps $\sqind{n}{m} : K_m \to_\pt K_{m+n}$ for $m,n \geq 0$, called the Steenrod squares, 
  usually written $\sq{n}$ leaving the $m$ implicit, which satisfy the following identities.
  \begin{itemize}[align=parleft, labelsep=.7cm]
  \myitem{\textnormal{(I1)}} \,$\sq{0}(x) = x$ \label{ax1}
  \myitem{\textnormal{(I2)}} \,$\sqind{n}{m}(x) = 0$ if $n > m$ \label{ax2}
  \myitem{\textnormal{(I3)}} \,$\sqind{n}{n}(x) = {x}\smile{x}$ \label{ax3}
  \myitem{\textnormal{(C)}} \,$\sq{n}(x\! \smile \!y) = \!\!\!\!\underset{i + j = n}{\sum}\!\!{\sq{i}(x)\! \smile \!\sq{j}(y)}$ (the Cartan formula) \label{ax4}
  \end{itemize}
  In addition, the squares are stable and satisfy the Adem relations:
  \begin{itemize}[align=parleft, labelsep=.7cm]
  \myitem{\textnormal{($\Omega$)}} The $n$th square $\sqind{n}{m} : K_m \to_\pt K_{m+n}$ is also given by \label{ax:susp}
    \begin{center} ${K_m \xrightarrow{\sim} \Omega(K_{m+1}) \xrightarrow{\Omega(\sqind{n}{m+1})} \Omega(K_{(m+1)+n}) \xrightarrow{\sim} K_{m + n}}$.
    \end{center}
  \myitem{\textnormal{(A)}} The Adem relations are satisfied: for $n<2k$, we have \label{ax:adem}
    \begin{center}$
	  \sq{n} \circ \sq{k} = \sum^{\lfloor{n/2}\rfloor}_{i = 0} {{k-i-1}\choose{n-2i}}\sq{n+k-i}\circ\sq{i}.
	  $\end{center}
  \end{itemize}
\end{theorem}
In \ref{ax3} and \ref{ax4} above, the symbol $\smile$ denotes the cup product map 
$K_m \to_\pt K_n \to_\pt K_{m+n}$ obtained from
the ring structure on $\bZtwo$.
\subsection{Contributions and outline}
The main contribution of this paper is a proof of the properties of Steenrod
squares listed in \cref{thm:main}. Before discussing Steenrod squares, we discuss, in
\cref{sec:pairs}, some background material about $\RPinf$ and
unordered pairs. Then in \cref{sec:maindef}, we present a variant of
Brunerie's definition of the Steenrod squares, via what we call unordered cup products.
We prove all the properties of Steenrod squares listed in \cref{thm:main}, 
modulo a technicality dealt with in the following section. 
\cref{sec:fubini} concerns properties of unordered HITs and constitutes 
the technical core of this paper.
In \cref{sec:Einf} we discuss how a good theory of $E_\infty$-monoids
in type theory would have significantly simplified our work.
In \cref{sec:applications} we showcase an application of Steenrod squares toward
analysing $\pi_4(\sphere{3})$.
Finally, in \cref{sec:conclusion}, we mention some open questions.

\subsection{Notation and basic definitions}
Let us briefly introduce some notation while also recalling some basic constructions from HoTT which will be used in the paper. The reader familiar with HoTT should be able to safely skim or even skip this part. 

\paragraph{Pi- and sigma-types} Let $B:A \to \UU$ be a dependent type -- here $\UU$ denotes the universe of types (at some implicit universe level). We often use the notation $(a : A) \to B\,a$ and $(a : A) \times B\,a$ for $\Pi_{a:A}B(a)$ and $\Sigma_{a:A}B(a)$ respectively.
We may sometimes write $B^A$ for the non-dependent function type $A \to B$.

\paragraph{Equality} We write $x = y$ for the type of paths from $x$ to $y$. We use $x \coloneqq y$ for definitions. The constant path (reflexivity) is denoted by $\refl_x : x = x$. We use \emph{path induction} to refer to the usual induction principle for identity types in MLTT.

\paragraph{Equivalences and univalence} 
A type $X$ is said to be \emph{contractible} if there is some $x:X$ s.t. for all $x': X$, we have that $x' =x$. Given a map $f: A \to B$ we define its fibre over some point $b:B$ by $\fib_f(b) \coloneqq (a : A)\times(f(a) = b)$. We say that $f$ is an equivalence if $\fib_f(b)$ is contractible for each $b:B$. In this case, we simply write $f:A \simeq B$ and leave the contractibility proof implicit. 

There is a canonical map $\mathsf{coe} : X = Y \to X \simeq Y$ defined by path induction, sending $\refl_X$ to the identity $\mathsf{id}_X : X \simeq X$. The univalence axiom says that $\mathsf{coe}$ itself is an equivalence. In particular, this means that if two types are equivalent, then they are equal.

\paragraph{Pointed structures} A pointed type $(A,\pt_A)$, i.e.\ a type $A$ equipped with a basepoint $\pt_A:A$, will often simply be written $A$, i.e.\ with the basepoint left implicit. We use the same convention for pointed functions and simply write $f : A \to_\pt B$ to mean a pair $(f,\pt_f)$ where $f: A \to B$ is a plain function and $\pt_f : f(\pt_A) = \pt_B$ is a proof that $f$ is basepoint preserving.

\paragraph{Loop spaces} We define the \emph{loop space} of a pointed type $A$ by $\Omega(A) \coloneqq (\pt_A = \pt_A)$. This construction is itself pointed by $\refl_{\pt_A}$ and can thus be iterated by inductively defining $\Omega^{n+1}(A) \coloneqq \Omega^{n}(\Omega(A))$.

\paragraph{H-levels} We say that a type $A$ is a \emph{$(-2)$-type} if it is contractible and, inductively, that it is an $n$-type if $x = y$ is an $(n-1)$-type for all $x,y:A$. Apart from $(-2)$-types, special names are also given to $(-1)$-types and $0$-types; these are, respectively, called \emph{propositions} and \emph{sets}.

\paragraph{Truncations} We write $\truncT{n}{A}$ for the \emph{$n$-truncation of $A$}, the canonical way of forcing $A$ to become an $n$-type. This is a type equipped with an inclusion of points $\trunc{-} : A \to \truncT{n}{A}$ whose induction principle say that the map $((x : \truncT{n}{A}) \to B(x)) \to ((a : A) \to B\trunc{a})$ is an equivalence whenever $B: \truncT{n}{A} \to \UU$ is a family of $n$-types. It is implemented using a recursive HIT~\cite[Section 7.3]{HoTT13}, but we do not need the implementation details here.

\paragraph{Connectedness} We say that a type is $n$-connected if $\truncT{n}{A}$ is contractible. We say that a function $f: A \to B$ is $n$-connected if all of its fibres are $n$-connected.

\paragraph{Pushouts} 
Given a span $Y \xleftarrow{f} X \xrightarrow{g} Z$, we define its (homotopy)
pushout, $Y \sqcup^X Z$, to be the HIT generated by two point constructors
$\inl{} : Y \to Y \sqcup^X Z$ and $\inr{} : Z \to Y \sqcup^X Z$, as well as
one higher constructor $\mathsf{push} : (x : X) \to \inl{(f\,x)} =
\inr{(g,x)}$. An important special case of pushouts is the
\emph{suspension} of a type $X$, written $\Sigma X$, which we define by
$\Sigma X \coloneqq \bOne \sqcup^X \bOne$. Another important construction is
the \emph{smash product} of two pointed types, denoted $X \wedge Y$. We
define it here as the pushout $(\mathbbm{1} + \mathbbm{1}) \sqcup^{X + Y}
X\times Y$. 
The \emph{join} of types is another example of a pushout which plays a
key role for us. The join of $X \ast Y$ of types $X, Y$ is
defined to be the pushout $X \sqcup^{X \times Y} Y$ of the
span
$X \xleftarrow{\mathsf{fst}} X \times Y \xrightarrow{\mathsf{snd}} Y$.
We take all pushouts to be pointed, whenever possible, by $\inl(\pt_X)$.

\paragraph{Eilenberg--MacLane spaces and cohomology} 
Given an abelian group $G$ and a natural number $n$, we denote by $K(G,n)$ the $n$th 
\emph{Eilenberg--MacLane space}, or \emph{delooping}, of $G$~\cite{LicataFinster14}. 
It is characterised as the unique pointed $(n-1)$-connected $n$-type
whose $n$th loop space $\Omega^n K(G,n)$ is isomorphic, as a group, to $G$.
It follows that we have pointed equivalences $\sigma_n : K(G,n) \simeq_\pt \Omega K(G, n+1)$.
In this way it is clear that $K(G,n)$ has an associative and commutative H-space structure
(corresponding to path composition in $\Omega K(G, n+1)$), denoted
$+ : K(G, n) \to K(G, n) \to K(G, n)$.

If moreover $R$ is a ring, then we have a \emph{cup product} $\smile : K(R, n)
\to_\pt K(R, m) \to_\pt K(R, n+m)$ 
which is graded-commutative and associative.
The defining property of the cup product is that in degree $(0,0)$, i.e.\ as
a map $K(R,0) \to K(R,0) \to K(R,0)$, it simply corresponds to multiplication in $R$,
and that the cup product respects looping in the following sense. For a fixed $a : K_n$,
consider the pointed map given by cupping with $a$, i.e.\ $(-)\smile_m a : K_m \to_\pt K_{m+n}$.
We have $\Omega ((-) \smile_{m+1} a) \circ \sigma_m = \sigma_{m+n} \circ ((-)\smile_m a)$. This is proved in detail in~\cite[Lemma 29]{LM24}.

The $n$th cohomology group of a type $X$ with coefficients in an abelian group $G$ is given by $H^n(X,G) \coloneqq \truncT{0}{X \to K(G,n)}$. 
The types $H^\bullet(X,G)$ are abelian groups by pointwise addition in
$K(G,n)$, and moreover form a graded ring if $G$ has the structure of a ring.






\section{Unordered pairs and commutativity structures}
\label{sec:pairs}
One of the main themes in this paper, crucial for the treatment of Steenrod squares,
is that of constructions that depend on 2-element type.
Recall that we write $\two$ for the standard 2-element type, with two distinct elements $0$ and $1$.
A general type $X$ is said to be a 2-element type if we have $\ptrunc{X \simeq \two}$.
Thus any 2-element type is merely equivalent to $\two$, but it has no preferred enumeration.
Recall also that we write $\RPinf$ for the type of all 2-element types, or more explicitly
\[
	\RPinf \coloneqq (X : \UU) \times \ptrunc{X \simeq \two}.
\]
We treat $\RPinf$ as a pointed type with basepoint $(\two, \trunc{\mathsf{id_{\two}}})$.
Buchholtz and Rijke \cite{realproj} have shown that $\RPinf$ is the sequential colimit of the finite-dimensional
real projective spaces $\RP^n$, hence the notation.
Given a term $X : \RPinf$ we will conflate $X$ with its underlying type $\mathsf{fst}(X) : \UU$.

For us, $\RPinf$ is significant because it captures the idea of commutativity in a synthetic
and homotopy coherent manner.
Consider, for example, the symmetry of cartesian products, $A_0 \times A_1 \simeq A_1 \times A_0$.
This can be explained as follows using $\RPinf$.
The binary cartesian product can be seen as a $\two$-indexed dependent product,
$\Pi_{i : \two} A(i)$ for $A : \two \to \UU$.
More generally, for any $X : \RPinf$ and $A : X \to \UU$, we can consider the
product $\Pi_{i : X} A(i)$. Clearly this reduces to the binary cartesian product
if $X$ is $\two$. Now $\two$ has a self-equivalence $\neg : \two \to \two$
which swaps $0$ and $1$, and by univalence this induces a loop $\pt = \pt$
of the basepoint of $\RPinf$.
By action on paths, this induces an equivalence
$\Pi_{i : \two} A(i) \simeq \Pi_{i : \two} A(\neg i)$ for any $A : \two \to \UU$,
which reduces to commutativity of the cartesian product.

In general, our usage of $\RPinf$ follows the same pattern. The point is to
first generalise some construction which normally would operate on ordered
pairs -- like the cartesian product $A_0 \times A_1$ computed from the pair
$(A_0, A_1) : \UU \times \UU$ -- to a construction that depends on an
`unordered pair'. By an unordered pair (of elements of $A$), we mean a map 
$a : X \to A$ where $X : \RPinf$ and $A$ is some arbitrary type.
We often write $A^X$ to emphasise that it should be thought of as a generalisation
of $A^2$ -- one might also say that $A^X$ is a `twisted' version of $A^2$.

The upshot is that whenever we write down a construction indexed by $\RPinf$,
we automatically gain information about the symmetry of said construction.
In the case of the cartesian product, this gives rise to the equivalence
$\mathsf{swap}_{A_0,A_1} : A_0 \times A_1 \simeq A_1 \times A_0$, but this is not all.
The fact that the self-equivalence $\neg : \two \simeq \two$ is involutive
tells us that $\mathsf{swap}_{A_1,A_0} \circ  \mathsf{swap}_{A_0,A_1} = \mathsf{id}$.
This is only the start of an infinite tower of coherences, associated with the cell decomposition
of $\RPinf$. The upshot of the synthetic approach is that we do not need to think explicitly
about these coherences.

\subsection{Basic facts about $\RPinf$}

Let us now recall some elementary lemmas and constructions regarding $\RPinf$
and unordered pairs. The following lemma is a special case of a result due to
Kraus~\cite[Proposition 8.1.2.]{Kraus15} and the remaining ones can be found in
Buchholtz and Rijke~\cite{realproj}.
\begin{lemma}\label{lem:RPinf-elim}
  For $P : \RPinf \to \UU$, the type of functions $(X : \RPinf) \to P(X)$ is equivalent to
  \begin{enumerate}[label={{(\alph*)}}]
  \item \label{lem:RPinf-elim-prop} $P(\two)$ if $P$ is proposition-valued.
  \item \label{lem:RPinf-elim-set} $(t : P(\two)) \times (P(\neg)(t) = t)$ if $P$ is set-valued.
\end{enumerate}
\end{lemma}
\begin{lemma}\label{lem:RPinf-set}
  All types $X:\RPinf$ are sets.  
\end{lemma}
\begin{lemma}\label{lem:PRinf-invol}
  For any $X:\RPinf$, there is an involution $\neg : X \simeq X$ agreeing with the usual involution of $\two$ whenever $X\coloneqq \two$. 
\end{lemma}
Although this lemma/definition is well known, its proof illustrates a useful technique for defining operations over $\RPinf$, so we choose to include it. 
\begin{proof}[Proof/construction of~\cref{lem:PRinf-invol}]
  For any $X:\RPinf$, let $P(X)\coloneqq (e : X \to X)\times((e \circ e = \mathsf{id}) \times e \neq \mathsf{id}_X)$. We claim that this type is contractible. Since this claim  is a proposition, it suffices, by \cref{lem:RPinf-elim}, to show that $P(\two)$ is contractible. This is trivial, as $P(\two)$ is the type of non-identity involutions on $\two$ -- a type uniquely pointed by $\two$-involution. So, $P(X)$ is contractible for any $X:\RPinf$ and we define $\neg$ to be the centre of contraction.
\end{proof}
Using \cref{lem:PRinf-invol}, we may construct, for any $X:\RPinf$ and $x:X$, an equivalence $e^x : \two \simeq X$ defined by setting
\[ e^x(0) \coloneqq  x \qquad e^x(1) := \neg x. \]
To prove that this indeed is an equivalence, we note that this statement is a proposition and thus, by \cref{lem:RPinf-elim}, it suffices to do so when $X$ is $\two$. By case-splitting on $x:\two$, we see that $e^x$ is the identity when $x = 0$ and involution when $x=1$. In particular, it is an equivalence. In fact, not only is $e^{x}$ always an equivalence -- the map  $e^{(-)}$ is one itself:
\begin{lemma}\label{lem:RPinf-charac}
  For any $X:\RPinf$, the map $e^{(-)} : X \to (\two \simeq X)$ is an equivalence.
\end{lemma}
\begin{proof}
  The statement is a proposition, and thus it suffices to show it when $X=\two$. In this case, $e^{(-)}$ is the map sending $0$ to the identity on $\two$ and sending $1$ to the involution. As these are precisely the (two) $\two$-automorphisms, $e^{(-)}$ is clearly invertible and thus an equivalence. 
\end{proof}
By univalence, \cref{lem:RPinf-charac} gives a characterisation of the based path types on $\RPinf$: it tells us that any based path type $(\pt_{\RPinf} = X)$ is equivalent to the `point' $X$ itself. In particular, we get that unordered pairs $(A^X)$ really corresponds to fibrations over the based path types of $\RPinf$, i.e. $(\pt_{\RPinf} = X \to A)$. This gives us a new way of interpreting path induction for $\RPinf$:
\begin{lemma}
  \label{lem:RPinf-ind}
  Let $A : (X : \RPinf) \times X \to \UU$. The map
  \[\left({((X,x) : (\dots))} \to A(X,x)\right) \xrightarrow{f \mapsto f(\two,0)} A(\two,0)\]
  is an equivalence.
\end{lemma}
Another way of understanding this is by the following induction rule for functions defined over $X:\RPinf$.
\begin{lemma}
  Let $X : \RPinf$, $B: X \to \UU$ and $x:X$. Any pair of points $b_0 : B(x)$ and $b_1 : B(\neg x)$ induces a function \[\pairfun {x \mapsto b_0}{\neg x \mapsto b_1} : (x : X) \to B(x)\] satisfying $\pairfun {x \mapsto b_0}{\neg x \mapsto b_1}(x) = b_0$ and $\pairfun {x \mapsto b_0}{\neg x \mapsto b_1}(\neg x) = b_1$. In fact, the map $B(x) \times B(\neg x) \xrightarrow{(b_0,b_1) \mapsto \pairfun {x \mapsto b_0}{\neg x \mapsto b_1}} \Pi_{x:X}B(x)$ is an equivalence.
\end{lemma}

\subsection{Commutativity Structures}
As discussed, the significance of unordered pairs is that they allow us to capture the idea
of an operation being homotopy commutative in an `infinitely coherent' manner.
This is captured by the following definition.

\begin{definition}[Brunerie~\cite{Brunerie17}]
  A \emph{commutativity structure} for a binary operation $\diamond : A \times A \to B$ is a family of maps $\diamond_X : A^X\to B$ for each $X:\RPinf$ agreeing with $\diamond$ if $X=\two$.
\end{definition}
By letting $A$ and $B$ be Eilenberg-MacLane spaces in the above definition, a commutativity structure $\diamond_{(-)}$ will allow us to produce cohomology classes in $H^*(\RPinf)$. Brunerie's construction of the Steenrod squares boils down to showing that the cup product $\smile : K_n \times K_n \to K_{2n}$ has a commutativity structure. Before we get there, however, let us give the following example in order to illustrate the general idea of how commutativity structures can be constructed. In fact, the following construction will be useful in its own right.  

\begin{example}\label{ex:+str}
  For any commutative monoid $(M,+,0)$, addition $+ : M \times M \to M$ has a commutativity structure. Since $M$ is a monoid, it is a set and thus the type of maps $M^X \to M$ is a set for any $X$. We will define the commutativity structure, denoted by $\Sigma : M^X \to M$ for $X:\RPinf$, using \cref{lem:RPinf-elim}{\ref{lem:RPinf-elim-set}}. For $X \coloneqq  \two$, we define $\Sigma f := f(0) + f(1)$. We then need to check that this definition is invariant under $\two$-inversion. This corresponds to checking that $f(0) + f(1) = f(1) + f(0)$ which of course follows from commutativity of $M$.
\end{example}
The construction in \cref{ex:+str} crucially relied on $M$ being a set; when constructing commutativity structures in general, there are not that many other methods than this at hand. Fortunately, this argument can sometimes still be used in cases when the h-level of the type of commutativity structures is not zero:
\begin{example}\label{ex:EM+str}
  Addition on Eilenberg--MacLane spaces  $+ : K_n \times K_n \to K_n$ has a commutativity structure. To see why, we consider the family of dependent types $P_X:(K_n^X \to K_n) \to \UU$  defined by $P_X(f) \coloneqq  (f(\lambda x\,.0) = 0)$. We have 
\[(f : K_n \times K_n \to K_n) \times (P_{\two}(f)) = (K_n \times K_n \to_\pt K_n)\]
which is a set~\cite[Corollary 9]{BDR18}. This means that \cref{lem:RPinf-elim}{\ref{lem:RPinf-elim-set}} applies. It thus suffices to provide an element of $P_{\two}(+)$ and check that this choice is invariant w.r.t.\ involution of $\two$. This boils down to verifying the commutativity of $+$.
\end{example}
This idea of defining a predicate over the function type of interest which forces it to become a set is present also in Brunerie's original definition of a commutativity structure for the cup product. Although he does not state it exactly this way, Brunerie implicitly considers the following predicate. 
\begin{definition}\label{def:bihom-bad}
  Let $X:\RPinf$, $A : X \to \UU_\pt$ and $B:\UU_\pt$, and
  $f : (\prod_{x : X} A(X)) \to B$. We define $\mathsf{isBiHom}_X(f) : \UU$ to be the following type
  expressing that $f$ is `pointed in each argument':
  \begin{align*}
	&(f(\lambda\,x\,.\,\pt) = \pt) \times (\mathsf{pts} : B^X)
	\\
    &\!\!\times \! \left(\!\!\left(a : \prod_{x:X}A(x)\right)\!(x:X) \to (a(x) = \pt) \to f(a) = \mathsf{pts}(x)\!\right)
  \end{align*}
  Let $\mathsf{BiHom}_X(A , B) \!\coloneqq\! (f : \prod_{x:X} A(x) \to B)\times\mathsf{isBiHom}_X(f)$.
\end{definition}
A straightforward rewriting shows that, for any proof of $\mathsf{isBiHom}_X(f)$, its $\mathsf{pts}$ component
is constantly $\pt : B$, and we have
\[(f : A_0 \times A_1 \to B)\times \mathsf{isBiHom}_{\two}(f) \simeq (A_0 \wedge A_1 \to_\pt B).\]
By setting $A(x) = K_n$ and $B = K_{2n}$ in \cref{def:bihom-bad},  $\mathsf{isBiHom}_X$ is a predicate on the
function type $K^X_n \to K_{2n}$. Let us construct such a function.
The key observation is that the type of such functions is equivalent to $K_n
\wedge K_n \to_\pt K_{2n}$ whenever $X = \two$. This turns out to be a
set~\cite[Corollary 9]{BDR18}, and thus the type of such function is a set for
any $X:\RPinf$. Like in \cref{ex:+str,ex:EM+str}, it is enough to give the
construction when $X=\two$ and check that it is commutative. Since the cup
with coefficients mod 2 is commutative, it has a commutativity
structure.

This concludes (our take on) Brunerie's definition of the commutativity structure on the cup product. While it is certainly sufficient for constructing the Steenrod squares, it has turned out to be rather hard to reason about. One simple but crucial reason for this is that Brunerie's definition does not quite capture a key fact about the cup product, namely that it is \emph{graded}. The main issue we have is that our notion of a commutativity structure does not allow for dependent types. To remedy this, we propose the following definition.
\begin{definition}\label{def:graded-comm-str}
  Let $A : I \to \UU$ be a family of types where $I$ is a commutative monoid (e.g. $I=\bN$). A graded commutativity structure for a graded operation $\diamond : A_i \times A_j \to A_{i+j}$ is a family of maps $\diamond_{X,n} : (\Pi_{x:X}A_{n(x)}) \to A_{\Sigma n}$ for each $X:\RPinf$ and $n : X \to I$,
which reduces to $\diamond$ for $X \coloneqq \two$. We remind the reader of the definition of $\Sigma n$ from~\cref{ex:+str}.
\end{definition}

\begin{remark}
Since we work modulo 2 throughout, we have no reason to worry about the signs that normally show up when
discussing graded commutativity of e.g\ the cup product, but let us make a comment about how they
can be dealt with. For a group $G$ and a finite type $\mathbf{n}$ of $n$ elements, one can define a pointed
type $K(G, \mathbf{n})$, which is like $K(G, n)$ but with a `twist' relating
odd permutations of $\mathbf{n}$ with the involution of $K(G,n)$ given by negation.
If $G$ then is a commutative ring, the corresponding graded commutativity structure
on $K(G, -)$ would be given by maps
$\Pi_{x : X} K(G, \mathbf{n}(x)) \to K(G, \sum_{x : X} \mathbf{n}(x))$ for $X : \RPinf$
and $\mathbf n : X \to \UU$ a family of finite types.
The key here is to index not by $\bN$ but by the type of finite sets, or some other higher type
which records information about twists.
\end{remark}

Our construction of the commutativity structure on the cup product can be restated, word by word, to equip it with a graded commutativity structure. This slight generalisation of Brunerie's definition will be the one used in this paper. For this reason, let us finish this section by giving it a name.
\begin{definition}
  The cup product has a graded commutativity structure which we will refer to as the \textbf{unordered cup product}. For $X:\RPinf$, $n:X\to \bN$ and $f : (x:X) \to K_{n(x)}$, we denote it by $\bigsmile_{x:X}f(x) : K_{\Sigma n}$. 
\end{definition}



\section{The Steenrod squares}
\label{sec:maindef}
\normalsize

We are now well-prepared to define the Steenrod squares. We follow Brunerie's
approach, as laid out in \eqref{eq:sqdef}. That is, we need to define two maps:
one of type $K_m \to (\RPinf \to K_{2m})$ and one (equivalence) of type
$(\RPinf \to K_{2m}) \xrightarrow{\sim} \prod_{i \leq 2m}K_{i}$. Let us start with the first
map. 
Suppose we are given $a : K_m$ and $X:\RPinf$. We let $n: X \to \bN$ and
$\widehat{a} : (x : X) \to K_{n(x)}$ be the constant functions $n(x) \coloneq
m$ and $\widehat{a}(x) := a$. We may now define $a^X : K_{2m}$ by
\begin{equation}
\label{eq:S}
  a^X \coloneqq \bigsmile_{x:X}\widehat{a}(x).
\end{equation}
The notation is meant to suggest that we think of $a^X$ as the cup product of $X$-many copies of $a$;
traditionally, this may also be written as $S(a,X)$.

The second map is given by the inverse equivalence in the following lemma.
\begin{lemma}\label{lem:polynomial-representation}
For $n : \bN$, we have an equivalence
\begin{align*}
\gysin_n: \prod_{i \leq n}K_i
&\simeq
(\RPinf \to K_n)
	\\
		\gysin_n(b_0, \ldots, b_n) &\coloneqq  
X \mapsto
\sum_{i = 0}^n b_i \smile t(X)^{n-i}
.
\end{align*}
\end{lemma}
Here, $t$ denotes the unique pointed equivalence $\RPinf \xrightarrow{\sim}_\pt K_1$, and $t(X)^{n-i}$ denotes the
iterated cup product of $t(X)$ with itself.
One can think of~\cref{lem:polynomial-representation} as describing the mod 2 cohomology of $\RPinf$, but
more directly it says that every map $\RPinf \to K_n$ has a unique `polynomial' representation.
Before proving~\cref{lem:polynomial-representation}, we first have to state two simpler lemmas.
\begin{lemma}\label{lem:gysin}
For $n : \bN$, we have an equivalence
\begin{align*}
	(\RPinf \to K_n) &\simeq (\RPinf \to_\pt K_{n+1}) \\
		f &\mapsto X \mapsto t(X) \smile f(X)
\end{align*}
\end{lemma}
\cref{lem:gysin} is proved using the Thom isomorphism~\cite[Section 6.1]{Brunerie16}. For details, see~\cite[Section 5.5]{LM24}.
\begin{lemma}\label{lem:h-space-pt-map}
For any invertible H-space $B$ and pointed type $A$, we have an equivalence
\begin{align*}
B \times (A \to_\pt B) &\simeq A \to B
\\
(b , f) &\mapsto a \mapsto b + f(a)
\end{align*}
\end{lemma}
\begin{proof}
  We have \begin{align*}
    B \times (A \to_\pt B) &\simeq (b : B) \times (A \to_\pt (B,b)) \\
    &\simeq (f : A \to B) \times (b : B) \times (f(a) = b)\\
    &\simeq (A \to B)
    \end{align*}
where the first step comes from the fact that $B$ is an invertible H-space (and hence homogeneous) and the second from the contractibility of singletons. This equivalence agrees with the proposed one by construction.  
\end{proof}
We are now ready to prove~\cref{lem:polynomial-representation}.
\begin{proof}[Proof of~\cref{lem:polynomial-representation}]
By induction.
For $n = 0$, this is simply the statement that any map $\RPinf \to K_0$ is constant, which
follows from connectedness of $\RPinf$.
Now, suppose the lemma holds for some $n \ge 0$.
Then
\begin{align*}
\Pi_{i \le {n+1}} K_i & \simeq
K_{n+1} \times \Pi_{i \le n} K_i
\\
&\simeq K_{n+1} \times (\RPinf \to K_n)
\\
&\simeq K_{n+1} \times (\RPinf \to_\pt K_{n+1})
\\
&\simeq \RPinf \to K_{n+1}.
\end{align*}
Here, the second line is by the inductive hypothesis, the third by \cref{lem:gysin}, and the final
line is by \cref{lem:h-space-pt-map}. It is direct to see that the forward composite is
the desired one.
\end{proof}
Finally, we are ready to define the Steenrod squares.
\begin{definition}[Steenrod squares]\label{def:squares}
  We define the total square $\totSq : K_m \to \prod_{i\leq 2m}K_{i}$ by $\totSq(a) \coloneqq  \gysin^{-1}_{2m}(a^{(-)})$.We define the $n$th Steenrod square $\sq{n} : K_m \to K_{m+n}$ by
  \begin{align*}
    \sq{n}(a) \coloneqq  \begin{cases} \mathsf{proj}_{m+n}(\totSq(a)) &\text{ if $n \leq m$} \\ 0 &\text{ otherwise} \end{cases}
  \end{align*}
\end{definition}
Unpacking the definition, we get the following characterisation of $\sq{n}(a)$ for a given $a : K_m$: they are the unique collection of terms such that for every $X :\RPinf$ we have%
\footnote{Formally, we should include terms $i$ from $0$ to $2m$ in the equation.
	But the corresponding maps $\sq{n} : K_m \to K_{m+n}$ with $n < 0$ are zero for connectedness reasons;
		they are pointed by \cref{lem:ptd}.}
\begin{equation}\label{eq:sq-poly}
	a^X = \sum_{i=0}^m \sq{m}(a) \smile t(X)^{m-i}.
\end{equation}
Note that \ref{ax2} holds by construction with this definition of $\sq{n}$.

\subsection{Proving the main theorem}
Now that we have a definition of the Steenrod squares (following Brunerie), let us, in this section, work our way towards a proof of \cref{thm:main}. 
The idea is to use \cref{eq:sq-poly} to reduce properties of $\sq{n}$ to properties of $(-)^X$, and hence
of the unordered cup product. 
The following is a simple example.
\begin{lemma}\label{lem:ptd}
The Steenrod squares are pointed, i.e.\ $\sq{n}(0) = 0$.
\end{lemma}
\begin{proof}
We have $0^X = 0$ for any $X : \RPinf$ since the unordered cup product is a bihom by construction.
Thus \cref{eq:sq-poly} gives $0 = \sum_{i = 0}^n \sq{n}(a) \smile t(X)^{n-i}$ for all $X : \RPinf$.
By \cref{lem:polynomial-representation}, we must have $\sq{n}(a) = 0$ for all $n$.
\end{proof}
Perhaps more interestingly, the Cartan formula is equivalent to the following
innocuous equation:
\begin{equation}\label{eq:cartan-poly}
  (a \smile b)^X = a^X \smile b^X
  .
\end{equation}
We will prove the above equation via the following generalisation, which can be thought
of as a type of Fubini interchange law and also will give rise to the Adem relations.
\begin{theorem}\label{thm:cup-fubini}
  For any $X,Y:\RPinf$, $n: X\times Y \to \bN$ and $f : \prod_{x:X}\prod_{y:Y}K_{n(x,y)}$, we have
  \[
  \bigsmile_{x:X}\bigsmile_{y:Y} f(x,y) = \bigsmile_{y:Y}\bigsmile_{x:X} f(x,y)
  .
  \]
\end{theorem}

While easy to state, proving \cref{thm:cup-fubini} is far more difficult than
anything we have done so far. Its proof, which we assume for now but which will
be discussed at length later in the paper, necessitates a development of the theory of
unordered joins and constitutes the technical core of this paper. Before we
are faced with reality, let us reap its fruits prematurely and prove the
characterising properties of the Steenrod squares laid out in \cref{thm:main}.

\begin{proposition}\label{prop:cartan}
  The Steenrod squares satisfy the Cartan formula~\ref{ax4}. 
\end{proposition}
\begin{proof}
Consider \cref{thm:cup-fubini} in these case where $Y$ is $\two$,
and $n$ and $f$ depend only their second arguments, so that they are given
simply by $i, j : \bN$ and $a : K_i$, $b : K_j$.
In this case, the $\two$-indexed `unordered' cup product reduces to the ordinary cup
product, and the $X$-indexed cup product reduces to $(-)^X$, so that
we end up with \cref{eq:cartan-poly}, $(a \smile b)^X = a^X \smile b^X$.
Combined with \cref{eq:sq-poly}, this gives the following:
\begin{align*}
	&\sum_{k=0}^{i+j} \sq{k}(a \smile b) \smile t(X)^{i+j-k}
	\\
	=  
	&\left(\sum_{l=0}^i \sq{l}(a) \smile t(X)^{i-l}\right)
			\smile
	\left(\sum_{m=0}^h \sq{m}(b) \smile t(X)^{j-m}\right)
	\\
	=  
	&\sum_{l=0}^i \sum_{m=0}^j \sq{l}(a) \smile \sq{m}(b) \smile t(X)^{i+j-l-m}
	.
\end{align*}
Since, for given $a$, $b$, the above identity holds in $K_{2(i+j)}$ for \emph{all} $X : \RPinf$,
	we may by \cref{lem:polynomial-representation} formally identify coefficients
	of the polynomials. This concludes the proof.
\end{proof}
\begin{proposition}\label{prop:square-cup}
  The Steenrod squares satisfy~\ref{ax3}: for $a : K_n$, we have $\sq{n}{(a)} = a \smile a$.
\end{proposition}
\begin{proof}
Taking $X$ to be $\two$ in \cref{eq:sq-poly}, we have $t(X) = 0$ and so only one term in the sum remains:
$a^{\two} = \sq{n}(a)$.
We have $a^{\two} = a^2$ since the unordered cup product generalises the ordinary cup product.
This concludes the proof.
\end{proof}
An important fact about Steenrod squares not listed in \cref{thm:main} is that they are additive:
$\sq{n}(a+b) = \sq{n}(a) + \sq{n}(b)$. This is a consequence of \ref{ax:susp},
essentially because the action of any function on paths respects path composition.
But there is also a more direct proof.
\begin{lemma}\label{lem:additivity}
The Steenrod squares are additive: for $a, b : K_m$ we have $\sq{n}(a+b) = \sq{n}(a) + \sq{n}(b)$.
\end{lemma}
\begin{proof}
By \cref{eq:sq-poly}, it suffices to prove that $(a+b)^X = a^X + b^X$. 
In fact, one can show a stronger statement, namely that $(-)^X : K_m \to K_{2m}$ has a unique delooping.
Since this is a proposition, we may suppose that $X$ is $\two$,
i.e.\ it suffices to show that $(-)^2 : K_m \to K_{2m}$ has a unique delooping.
By \cite[Corollary 12]{Warn2023}, the delooping is unique if it exists, and it exists if
and only if $(a+b)^2 = a^2 + b^2$. This holds by distributivity and commutativity since we are working mod 2.
\end{proof}
\begin{remark}
Related to \cref{lem:additivity}, we remark that an alternative, simpler definition
of $\sq{n}$ is possible. The stability axiom \cref{ax:susp} tells us that
the map $\sq{n} : K_m \to_\pt K_{m+n}$ should be a delooping of $\sq{n} : K_{m-1} \to_\pt K_{m+n-1}$.
By \cite[Corollary 12]{Warn2023} and the fact that $(a+b)^2 = a^2 + b^2$,
the delooping exists and is unique so we could take this as a recursive definition of $\sq{n}$,
starting from the definition of $\sq{n} : K_{n} \to K_{2n}$ as $x \mapsto x \smile x$.
In this way, one would immediately get a stable cohomology operation, which is sufficient for some purposes,
but this definition seems to give no insight toward proving the Cartan formula or Adem relations.
\end{remark}
Before we continue with the remaining axioms, we will need the following 
lemma which will allow a computation of $\sq{0}$ on $K_1$.
\begin{lemma}\label{lem:exp-self}
For $X : \RPinf$ we have $t(X)^X = 0$.
\end{lemma}
\begin{proof}
We have to show that
$\bigsmile_{x : X} t(X) = 0$. Since the unordered cup product has the structure
of a bihom, it suffices to prove that $t(X) = 0$ in $K_1$ for all $x : X$.
This is direct; given $x : X$ we indeed have $X \simeq \two$ so that $t(X) = 0$.
\end{proof}
The above proof may seem curious: we argue that $t(X)^X = 0^X$, not by showing 
that $t(X) = 0$, but by showing $t(X) = 0$ for all $x : X$.
\begin{lemma}\label{lem:sq-k1}
For $x : K_1$, we have $\sq{0}(x) = x$.
\end{lemma}
\begin{proof}
We have
$0 = t(X)^X = \sq{0}(t(X)) t(X) + \sq{1}(t(X))$ by \cref{eq:sq-poly}.
Here $\sq{0}$ is a pointed map $K_1 \to_\pt K_1$ so it is given by multiplication by some element
$c$ of $\bZtwo$. We have that $\sq{1}(t(X)) = t(X)^2$ by \cref{prop:square-cup}.
Thus 
$0 = c\,t(X)^2 + t(X)^2$. By formally identifying coefficients of polynomials, we
get $c = 1$, so that $\sq{0}$ is the identity map, as needed.
\end{proof}
In order to prove the stability axiom~\ref{ax:susp}, it will be helpful
to represent loop spaces in terms of maps from $\sphere{1}$.
In our setting, this is mediated by cup products and the pointed map 
$e : \sphere{1} \to_\pt K_1$ which sends the generating loop of $\sphere{1}$ to
the non-trivial loop of $K_1$, i.e.\ $\sigma_0(1)$,
according to the following lemma.
\begin{lemma}\label{lem:pytte}
With $e : \sphere{1} \to_\pt K_1$ as above and $n : \bN$, 
the composite of $\sigma_n : K_n \to \Omega K_{n+1}$ with
the canonical equivalence $\Omega K_{n+1} \to (\sphere{1} \to_\pt K_{n+1})$
is given by
$a \mapsto (x \mapsto e(x) \smile a)$.
\end{lemma}
\begin{proof}
Given $a : K_n$, it suffices to show that 
the action of $((-)\smile a) \circ e : \sphere{1} \to_\pt K_{n+1}$
on the generating loop, i.e.\ $\mathsf{loop} : \Omega \sphere{1}$, is given by $\sigma_n(a) : \Omega K_{n+1}$.
We have 
\begin{align*}
\Omega(((-)\smile a)\circ e)(\mathsf{loop}) 
		&= \Omega((-)\smile a)(\Omega(e)(\mathsf{loop}))
	\\
		&= \Omega((-)\smile a)(\sigma_0(1))
	\\
		&= \sigma_n (1 \smile a)
	\\
		&= \sigma_n (a)
\end{align*}
where, in the second-to-last step, we use the fact that the cup product respects looping.
\end{proof}

\begin{proposition}\label{prop:suspension}
  The Steenrod squares satisfy the stability axiom~\ref{ax:susp}.
\end{proposition}
\begin{proof}
Let  $\sqind{n}{m}$ denote the Steenrod square $K_m \to K_{m+n}$, 
let $\sigma_m : K_k \xrightarrow{\sim}_\pt \Omega K_{m+1}$ and let
$\tau_{m}$ denote the canonical equivalence
$\Omega K_m \xrightarrow{\sim} (\sphere{1} \to_\pt K_m)$.
Given $m,n : \bN$, we wish to show that square (A) in the following diagram commutes.
\[\begin{tikzcd}[ampersand replacement=\&]
	{K_m} \&\& {K_{m+n}} \\
	{\Omega K_{m+1}} \&\& {\Omega K_{m+n+1}} \\
	{(\sphere{1} \to_\pt K_{m+1})} \&\& {(\sphere{1} \to_\pt K_{m+n+1})}
	\arrow[""{name=0, anchor=center, inner sep=0}, "{\sqind{n}{m}}", from=1-1, to=1-3]
	\arrow["{\sigma_m}"', from=1-1, to=2-1]
	\arrow["{\sigma_{m+n}}", from=1-3, to=2-3]
	\arrow[""{name=1, anchor=center, inner sep=0}, "{\Omega \sqind{n}{m+1}}"{description}, from=2-1, to=2-3]
	\arrow["{\tau_{m+1}}"', from=2-1, to=3-1]
	\arrow["{\tau_{m+n+1}}", from=2-3, to=3-3]
	\arrow[""{name=2, anchor=center, inner sep=0}, "{\sqind{n}{m+1} \circ (-)}"', from=3-1, to=3-3]
	\arrow["{\text{(A)}}"{description}, draw=none, from=1, to=0]
	\arrow["{\text{(B)}}"{description}, draw=none, from=2, to=1]
\end{tikzcd}\]
To this end, we note that it is easy to see that square (B) commutes
and that all vertical maps are equivalences. Hence, it suffices to
show that the outer square commutes.
By \cref{lem:pytte}, the vertical composites are given by
$ a \mapsto (x \mapsto e(x) \smile a) $.
%
Thus it suffices to show that for every $a : K_m$ and $x : \sphere{1}$, we have
\[
\sq{n}(e(x) \smile a) = e(x) \smile \sq{n}(a)
	.
\]
By the Cartan formula, the left hand side computes to
$\sq{0}(e(x)) \smile \sq{n}(a) + \sq{1}(e(x)) \smile \sq{n-1}(a)$.
By \cref{lem:sq-k1} we have $\sq{0}(e(x)) = e(x)$ so it suffices
to show that $\sq{1}(e(x)) = 0$, i.e.\ that $e(x) \smile e(x) = 0$.

To see why this holds, one can simply note that the map $x \mapsto e(x) \smile e(x)$ factors as
\[
\sphere{1} \xrightarrow{\Delta} \sphere{1} \wedge \sphere{1} \xrightarrow{\smile} K_1 \wedge K_1.
\]
The diagonal map $\Delta : A \to A \wedge A$ vanishes whenever $A$ is
a suspension (so, in particular when $A = \sphere{1}$). This follows by straightforward suspension induction.
%
\end{proof}
\begin{proposition}\label{prop:zeroth-sq}
  The Steenrod squares satisfy axiom~\ref{ax1}, i.e.\ $\sq{0} = \mathsf{id}$.
\end{proposition}
\begin{proof}
The zeroth square lives in the type of pointed functions $K_n \to_\pt K_n$ -- a
type which we understand well: looping  $(K_n \to_\pt K_n) \to (K_{n-1} \to_\pt
		K_{n-1})$ is an equivalence for each $n\geq 0$. Since looping preserves
the identity function, it is thus, by~\ref{ax:susp}, enough to show that
$\sqind{0}{0} : K_0 \to K_0$ is the identity. By \ref{ax3}, we have
$\sqind{0}{0}(x) = x \smile x$. However, since $K_0 \coloneqq  \bZtwo$, the cup
product here is simply multiplication in $\bZtwo$ and thus $\sqind{0}{0}(x)
	= x \smile x = x$.
\end{proof}
\begin{proposition}\label{prop:adem}
  The Steenrod squares satisfy the Adem relations~\ref{ax:adem}. 
\end{proposition}
\begin{proof}
Given $m : \bN$ and $a : K_m$, consider \cref{thm:cup-fubini} in the case
where $n$ and $f$ are constantly $m$ and $a$.
In this case, we have for any $X, Y : \RPinf$ that
\[
(a^X)^Y = (a^Y)^X.
\]
The idea is now simply to expand each side using \cref{eq:sq-poly}.
We also make use of \cref{lem:additivity}, the Cartan formula, and the fact that
$t(X)^Y = t(X)^2 + t(X) \smile t(Y)$, which follows from \ref{ax3}.
In this way, we get the following.
\begin{align*}
(a^X)^Y &=
	\left( \sum_i \sq{i}(a) \smile t(X)^{n-i}\right)^Y
	\\
	&= \sum_i \sq{i}(a)^Y \smile (t(X)^2 + t(X) t(Y))^{n-i}
	\\
	&= \sum_{i,j,k} {n-i \choose k}\sq{j}\sq{i}(a)t(Y)^{2n-j-k} t(X)^{n+k-i}
\end{align*}
In the same way, one can express $(a^Y)^X$ as a polynomial in $t(X)$ and $t(Y)$.
Since we have $(a^X)^Y = (a^Y)^X$, we can formally identify coefficients in these polynomials,
by repeated application of \cref{lem:polynomial-representation}.
The steps required to go from here to the Adem relations are the same as in the classical case;
see \cite[Page 345]{hausmann2014mod} for details.
\end{proof}
This concludes the proof of~\cref{thm:main}. A natural question to ask after seeing the short and snappy proofs in this section is whether, perhaps, the setting of HoTT makes working with cohomology operations like the Steenrod squares `easier' than in a more traditional setting.
Although there are aspects of our setup here which certainly seem to speak in favour of HoTT, we wish to use this section to emphasise that \emph{a lot} of the heavy lifting is done by \cref{thm:cup-fubini} which we have, thus far, only assumed.
Thus, this question entirely hinges on the difficulty of proving this statement.
For this reason, we will now devote the remainder of the paper to proofs. They
are interesting in their own right, as they force us to develop a fair bit of
novel machinery surrounding unordered joins.
\section{Unordered joins and their Fubini theorem}
\label{sec:fubini}
We now set out to prove~\cref{thm:cup-fubini}. This theorem states that our
unordered cup product satisfies a certain `Fubini theorem'. 
Recall from the definition of the unordered cup product
$\smile_X : \Pi_{x : X} K_{n(x)} \to K_{\Sigma n}$
that it has the structure of a bihom, i.e.\ it is pointed in each argument,
and that it is non-trivial.
This characterises $\smile_X$ up to contractible choice, and
so anything we prove about $\smile_X$ should come from this characterisation.
In order to prove \cref{thm:cup-fubini}, we would like to have a similar characterisation
of the iterated cup product
\[
	\smile_X \smile_Y : \Pi_{x : X} \Pi_{y : Y} K_{n(x,y)} \to K_{\Sigma n}.
\]
In other words, we would like to find a way to uniquely characterise $\smile_X \smile_Y$,
that is symmetric in $X$ and $Y$.
Morally, this characterisation is that $\smile_X \smile_Y$ is coherently pointed
in each of its $X \times Y$-many arguments.
Thus we are lead to consider a generalisation of $\mathsf{isBiHom}$ to the 4-element
type $X \times Y$. Already defining such a generalisation is rather complicated; 
it involves the combinatorics of all non-empty (decidable) subsets of $X \times Y$ 
and their inclusions.
So it will be helpful to have another perspective on $\mathsf{isBiHom}$.
In fact, Brunerie~\cite{Brunerie17} never considered $\mathsf{isBiHom}$, and instead
worked with its corepresenting object, the unordered smash product:
\begin{definition}\label{def:unordered-smash}
  Let $X : \RPinf$ and $A : X \to \UU_\pt$ be a family of pointed types. We define the unordered smash product of $A$, denoted $\bigwedge_{x:X}A(x)$, by the following pushout
\[\begin{tikzcd}[ampersand replacement=\&,row sep = .4cm , column sep = 1.5cm]
	(x:X) \times A(x) \&\& \prod_{x : X} A(x) \\
	X  \&\& \bigwedge_{x:X}A(x)
	\arrow["\fst"',from=1-1, to=2-1]
	\arrow["(x{,}a) \,\mapsto\, {\pairfun{x \mapsto a}{ \neg x \mapsto \pt}}",from=1-1, to=1-3]
	\arrow[from=2-1, to=2-3]
	\arrow[from=1-3, to=2-3]
	\arrow["\lrcorner"{anchor=center, pos=0.125, rotate=180}, draw=none, from=2-3, to=1-1]
\end{tikzcd}\]
We take this type to be pointed by $\mathsf{inr}(\lambda x\,.\pt)$. 
\end{definition}
It is easy to see that
$\mathsf{isBiHom}(f)$ is equivalent to asking that $f$ factors through
$\inr : \Pi_{x : X} A(x) \to \bigwedge_{x : X} A(x)$ via a pointed map
$\bigwedge_{x : X} A(x) \to_\pt B$, and that
$\mathsf{BiHom}_X(A,B)$ is equivalent to the type of all such pointed maps
$\bigwedge_{x : X} A(x) \to_\pt B$.
In this way, one can see that $\smile_X \smile_Y$ is given by
the composite of the map $\Pi_{x : X} \Pi_y K_{n(x,y)} \to \bigwedge_{x : X} \bigwedge_{y : Y} K_{n(x,y)}$, given by $f \mapsto \inr\left(x \mapsto \inr(y \mapsto f(x,y))\right)$
with the unique non-trivial map 
$\bigwedge_{x : X} \bigwedge_{y : Y} K_{n(x,y)} \to_\pt K_{\Sigma n}$.
What remains to be shown is that we have a pointed equivalence $e : \bigwedge_{x:X}\bigwedge_{y:Y}A(x,y) \simeq \bigwedge_{y:Y}\bigwedge_{x:X}A(x,y)$ such that the following diagram commutes.
\[
\begin{tikzcd}[ampersand replacement=\&]
	{\Pi_{x:X}\Pi_{y:Y}A(x,y)} \&\& {\Pi_{y:Y}\Pi_{x:X}A(x,y)} \\
	{\bigwedge_{x:X} \bigwedge_{y:Y}A(x,y)} \&\& {\bigwedge_{y:Y} \bigwedge_{x:X}A(x,y)}
	\arrow["{\mathsf{swap}}", from=1-1, to=1-3]
	\arrow[from=1-1, to=2-1]
	\arrow[from=1-3, to=2-3]
	\arrow["e"', from=2-1, to=2-3]
\end{tikzcd}
\]
Again, this turns out to be rather complicated, and we need another simplifying device.
\begin{definition}\label{def:unordered-join} Let $X:\RPinf$ and $A: X \to \UU$. We define the unordered join of $A$, denoted $\bigast_{x:X}A(x)$, by the following pushout.
  \[
\begin{tikzcd}[ampersand replacement=\&]
	{X \times \Pi_{x:X}A(x)} \&\& {\Pi_{x:X}A(x)} \\
	{(x:X)\times A(x)} \&\& {\bigast_{x:X}A(x)}
	\arrow["\snd", from=1-1, to=1-3]
	\arrow["(x{,}f)\mapsto (x{,}f(x))"',from=1-1, to=2-1]
	\arrow[from=1-3, to=2-3]
	\arrow[from=2-1, to=2-3]
	\arrow["\lrcorner"{anchor=center, pos=0.125, rotate=180}, draw=none, from=2-3, to=1-1]
\end{tikzcd}
  \]
\end{definition}
The following lemma says that the unordered join agrees with the usual definition of joins when $X$ is $\two$. Recall that the usual definition is
$A_0 \ast A_1 \coloneqq A_0 \sqcup^{A_0 \times A_1} A_1$.
\begin{lemma}
  Given two types $A_0$ and $A_1$, we have \[\bigast_{x:\two}A_x \simeq A_0 \ast A_1.\]
\end{lemma}
\begin{proof}
By the $3 \times 3$ lemma applied to the following diagram.
\[\begin{tikzcd}[ampersand replacement=\&,cramped]
	{A_0} \& \varnothing \& {A_1} \\
	{A_0 \times A_1} \& \varnothing \& {A_0 \times A_1} \\
	{A_0 \times A_1} \& {A_0 \times A_1} \& {A_0 \times A_1}
	\arrow[from=1-2, to=1-1]
	\arrow[from=1-2, to=1-3]
	\arrow["{\mathsf{fst}}"', from=2-1, to=1-1]
	\arrow["{\mathsf{id}}", from=2-1, to=3-1]
	\arrow[from=2-2, to=1-2]
	\arrow[from=2-2, to=2-1]
	\arrow[from=2-2, to=2-3]
	\arrow[from=2-2, to=3-2]
	\arrow["{\mathsf{snd}}", from=2-3, to=1-3]
	\arrow["{\mathsf{id}}"', from=2-3, to=3-3]
	\arrow["{\mathsf{id}}"', from=3-2, to=3-1]
	\arrow["{\mathsf{id}}", from=3-2, to=3-3]
\end{tikzcd}\]
\end{proof}
The unordered join is relevant for us because $\mathsf{isBiHom}(f)$ is easily seen to be equivalent
to 
\[(a : \Pi_{x : X} A(x)) \to \bigast_{x : X} (a(x) = \pt) \to (f(a) = \pt) .\]
An equivalent way to phrase this (which we will not use)
is that $\bigwedge_{x : X} A(x)$ is the cofibre of the projection
\[(a : \Pi_{x : X} A(x)) \times \bigast_{x : X} (a(x) = \pt) \to
\Pi_{x : X} A(x)\] onto the first component; this map might be called an \emph{unordered wedge inclusion}.

Given this characterisation of $\mathsf{isBiHom}(f)$ in terms of the unordered join, the
proof of \cref{thm:cup-fubini} will eventually reduce to the following lemma.
\begin{lemma}\label{lem:fubini}
  For any $X,Y:\RPinf$ and $A:X\times Y \to \UU$, we have a function
  \[
  \bigast_{x:X}\bigast_{y:Y} A(x,y) \to \bigast_{y:Y}\bigast_{x:X} A(x,y)
  \]
\end{lemma}
Interestingly, we do not need to ask for any properties of this function, although
it is important that we \emph{construct} a function as opposed to proving its mere
existence.
Since the proof of \cref{lem:fubini} is rather technical, we postpone it for the moment
and turn to the main result promised at the beginning of this section.
\begin{proof}[Proof of \cref{thm:cup-fubini}]
Suppose we are given $X, Y : \RPinf$, $n : X \times Y \to \bN$, fixed
throughout the proof; we would like to show that $\smile_X \smile_Y = \smile_Y \smile_X$.
Given $f : \Pi_{x : X} \Pi_{y : Y} K_{n(x,y)}$, we claim that
we have a map
\begin{equation}\label{eq:blabla}
\bigast_{x : X} \bigast_{y : Y} (f(x,y) = 0) \to (\smile_{x : X} \smile_{y : Y} f(x,y) = 0)
.
\end{equation}
Indeed, we have a map
$\bigast_{x : X} \bigast_{y : Y} (f(x,y) = 0) \to \bigast_{x : X} \left(\smile_{y : Y} f(x,y) = 0\right)$
by functoriality of the unordered join together with the fact that $\smile_Y$ is a bihom,
   and a map
$(\bigast_{x : X} (\smile_{y : Y} f(x,y) = 0)) \to 
(\smile_{x : X} \smile_{y : Y} f(x,y) = 0)$ since $\smile_X$ is a bihom.

Now let $T$ be the sigma-type consisting of functions
$\mu : \Pi_{x : X} \Pi_{y : Y} K_{n(x,y)} \to K_{\Sigma n}$ together with a proof
that for all $f : \Pi_{x : X} \Pi_{y : Y} K_{n(x,y)}$,
	 we have $\bigast_{x : X} \bigast_{y : Y} f(x,y) = \pt \to \mu(f) = \pt$.
We can construct two elements of $T$: on the one hand we have $\smile_X \smile_Y$ together
with the argument above that we have a map as in \cref{eq:blabla}.
By symmetry and using \cref{lem:fubini}, we can similarly construct an element of $T$ 
whose first component is $\smile_Y \smile_X$.
Now we claim that these two elements of $T$ are equal; in fact, we claim that there
is a unique identification between them. This will finish the proof, since
if two pairs are equal then so are their first components.

Since our claim is now a proposition, we may assume that $X$ and $Y$ are both
$\two$. In this case,
	we have that $T$ is the set of pointed maps
	\[K_{n(0,0)} \wedge K_{n(0,1)} \wedge K_{n(1,0)} \wedge K_{n(1,1)} \to_\pt K_{\Sigma n} .\]
	By a version of `Cavallo's trick' \cite[Lemma 15]{Ljungstrom2024}, two pointed maps out of a smash product into a homogeneous type
	are equal if they are equal when restricted to the product,
	in this case $\Pi_{i,j\in \{0,1\}} K_{n(i,j)}$.
	Our two elements of $T$ correspond to the maps
	\[(a_{00},a_{01},a_{10}, a_{11}) \mapsto (a_{00} \smile a_{01}) \smile (a_{10} \smile a_{11})\]
	and
	\[(a_{00},a_{01},a_{10}, a_{11}) \mapsto (a_{00} \smile a_{10}) \smile (a_{01} \smile a_{11}) .\]
	Indeed these are equal by commutativity and associativity of the cup product.
	This concludes the proof.
\end{proof}
\subsection{Proving~\cref{lem:fubini}}
\todo{I'm just babbling here. Please rewrite.}
\todo{Här (eller direkt efter) hade en diskussion passat bra}

No result used in this project has turned out to be quite as problematic
as~\cref{lem:fubini} -- the result is highly technical and its computer
formalisation was only completed after a year's worth of failed attempts. While
we would be very happy to see a more conceptual proof of this statement, we are
sceptical to the existence of such a proof. To illustrate why, note that the
pushout describing the unordered join $\bigast_{x:X}A(x)$ in no way requires
$X$ to be a two-element type -- the definition makes sense for arbitrary $X$,
although in this case we might write it with an apostrophe $\bigast_{}'$
to remind ourselves that it is no longer a good generalisation of the usual join.
Thus, we may ask whether it is true, for \emph{any} two types $X$ and $Y$ and
family $A : X\times Y \to \UU$ whether the map

\[\mathcal{F} : \bigast_{x:X}{\!'}\bigast_{y:Y}{\!'}A(x,y) \to \bigast_{y:Y}{\!'}\bigast_{x:X}{\!'}A(x,y)\]
exists. This seems difficult (if not impossible) to do in general, as we cannot
prove in general that the domain and codomain are equivalent. For instance, if
we set $X=\two$ and $Y = \mathsf{hProp}$, $A(0,P)\coloneqq  \neg P$ and
$A(1,P) := P$, then the LHS becomes contractible whereas the RHS is equivalent
to the suspension of LEM -- a type whose contractibility is
independent of HoTT.


So, let us try to define $\mathcal{F}$ for $X,Y:\RPinf$. The function will be described by the following data:
\begin{align*}
  \mathcal{F}_l &: \left(\prod_{x:X}\bigast_{y:Y}A(x,y)\right) \to \bigast_{y:Y}\bigast_{x:X}A(x,y)\\
  \mathcal{F}_r &: (x:X)\times \bigast_{y:Y}A(x,y) \to \bigast_{y:Y}\bigast_{x:X}A(x,y)\\
  \mathcal{F}_{lr} &: (x : X)\left(f : \prod_{x:X}\bigast_{y:Y}A(x,y)\right) \to F_l(f) = F_r(x,f(x))
\end{align*}
The key problem here is defining $F_l$ -- its codomain is a $\Pi$-type and thus does not automatically come equipped with an elimination rule. Consequently, we need to understand the type $\prod_{x:X}\bigast_{y:Y}A(x,y)$. Luckily, it turns out that we can describe this $\Pi$-type using a rather involved construction. To give it a (somewhat) more concise definition, let us define, for any two types $B$ and $C$, and family $R : B \to C \to \UU$, the \emph{relational} pushout, $P(B,C,R)$ to simply be the pushout of the span $B \leftarrow (b:B)\times(c:C)\times R(b,c) \rightarrow C$. Any pushout $B \xrightarrow{f} D \xrightarrow{g} C$ can be written as a relational pushout by setting $R(b,c) \coloneqq  (d:D) \times (f(d) = b) \times (g(d) = c)$ (and vice versa).

\begin{lemma}\label{lem:HIT}
Let $X : \RPinf$, $B, C : X \to \UU$, 
and $R : B(x) \times C(x) \to \UU$ (with $x:X$ an implicit argument).
Then $\prod_{x:X}P(B(x), C(x), R(x))$
is equivalent to the HIT $T$ with the following constructors.
\begin{itemize}
\item[\small{$bb:$}] {\small $(\prod_{x : X} B(x)) \to T$}
\item[\small{$cc:$}] {\small $(\prod_{x : X} C(x)) \to T$}
\item[\small{$bc:$}] {\small $(x : X) \times B(x) \times C(\neg x) \to T$}
\item[\small{$rr:$}] {\small $(b : \prod_{x : X} B(x)) (c : \prod_{x : X} C(x))
		(r : \prod_{x : X}  R(b(x),c(x))) \to bb(b) = cc(c)$}
\item[\small{$br:$}] {\small $(x : X)  (b : \prod_{k : X} B(k)) (c : C(\neg x)) (r: R(b(\neg x),c)) \hspace{.3cm} \to
	bb(a) = bc(b(x),c)$}
\item[\small{$cr:$}] {\small $(x : X) (b : B(x)) (c : \prod_{k : X} C(k)) (r : R(b,c(x))) \to
	bc(b,c(\neg x)) = cc(c)$}
\item[\small{$rr':$}] {\small $(b : \prod_{x : X} B(x))\, (c : \prod_{x : X} C(x)) \hspace{3cm} \to 
  (r : \prod_{x : X}  R(b(x),c(x)))\,(x : X) \hspace{3cm} \to rr(b,c,r) \!=\! br(x,\neg x,b,c(\neg x),r(\neg x))\hspace{-.3mm} \cdot \hspace{-.3mm} cr(x,\neg x,b(x),c,r(x))$
}
\end{itemize}
\end{lemma}
\normalsize
\begin{proof}[Proof sketch]
  It is straightforward to define a map $w_X : T \to \Pi_{x:X} P(B(x),C(x),R(x))$ by $T$-induction. By \cref{lem:RPinf-elim}\ref{lem:RPinf-elim-prop}, it is sufficient to show that $w_{X}$ is an equivalence when $X=\two$. This is somewhat technical but can be done with relative ease.
\end{proof}
If we instantiate the above with $B(x) \coloneqq  (y:Y)\times A(x,y)$, $C(x):= \prod_{y:Y}A(x,y)$ and $R((y,a),f) := (f(y) = a)$, we have $P(B(x),C(x),R(x)) \simeq \bigast_{y:Y}A(x,y)$ and thus \cref{lem:HIT} tells us that there is an equivalence $w : T \simeq \prod_{x:X}\bigast_{y:Y}A(x,y)$. Hence, in order to construct $\mathcal{F}_l$, it is enough to define a map $T \to \bigast_{y:Y}\bigast_{x:X}A(x,y)$.

Mapping out of $T$ is, in general, not much easier than mapping out of $\prod_{x:X} \bigast_{y:Y}A(x,y)$:  a map out of $T$ must be defined over $\Pi_{x:X}B(x)$ and $\Pi_{x:X}C(x)$ which again forces us to map out of $\Pi$-types. The type $\Pi_{x:X}C(x)$ is unproblematic: it is simply $\Pi_{x:X}\Pi_{y:Y} A(x,y)$ and defines an element of $\bigast_{y:Y}\bigast_{x:X}A(x,y)$ by simply swapping the arguments. However, $\prod_{x:X}B(x) \coloneqq  \Pi_{x:X}((y:Y)\times A(x,y))$ is more complicated -- where we send an element $f$ of this type depends on the behaviour of $\fst \circ f : X \to Y$. Fortunately, we can understand this type.
\begin{lemma}\label{lem:fun-charac}
  For any $X,Y:\RPinf$, the map $((X\simeq Y) + Y) \to (X\to Y)$ sending equivalences to their underlying functions and $y:Y$ to the constant map $\lambda\, x \,.\,y$ is an equivalence.
\end{lemma}
\begin{proof}
  Since the statement is a proposition, it suffices to show it when $X=Y=\two$. In this case, the statement is simply the trivial observation that any function $\two \to \two$ is either an equivalence or constant.
\end{proof}
Using \cref{lem:fun-charac}, we can replace each occurrence of $\Pi_{x:X}((y:Y)\times A(x,y))$ in $T$ with the equivalent type
  \[ \left((e : X \simeq Y)\times A(x,e(x))\right) + \left((y : Y) \times A(x,y)\right) \]
which has a more well-behaved elimination principle. After this rewriting, it is possible to define, by pattern matching,  a map $\psi_{X,Y} : {T} \to \bigast_{y:Y} \bigast_{x:X}A(x,y)$, which allows us to define $\mathcal{F}_l = \psi_{X,Y} \circ \phi$, where $\phi$ is the appropriate instance of \cref{lem:HIT}. We then need to define, for each $x:X$, the function $\mathcal{F}_r(x,-)$ and the homotopy $\mathcal{F}_{lr}(x,-)$. In theory, this is somewhat easier: as we have $x:X$ in context, we may apply~\cref{lem:RPinf-ind}. In practice, we have to deal with a large number of difficult coherence problems which we completely sweep under the rug here.

\section{Joins and $E_\infty$-monoids}
\label{sec:Einf}
Much progress in synthetic homotopy type theory is held back by what is known
as the \emph{problem of infinite objects}~\cite{Buchholtz19}. This problem
appears in many guises, and is traditionally explained in terms of
semi-simplicial types. In this work we brush against instances of this
problem in many places where we would like to talk about higher commutative, more
precisely $E_\infty$-, monoids.

The idea of $E_\infty$-monoids%
\footnote{Usually referred to by the less descriptive term $E_\infty$-\emph{spaces}.}
is remarkably natural from a type-theoretic perspective.
An $E_\infty$-monoid should consist of a type $A$ and for \emph{any finite type} $X$
a `multiplication' map
\[\mu_X : A^X \to A .\]
These multiplication maps express that any finite collection of elements of $A$ can be multiplied,
and that their order is irrelevant in a homotopy coherent sense.
The unary multiplication $\mu_1 : A^1 \to A$ should be the canonical equivalence.
These multiplication maps are subject to certain coherences, starting with the following:
given a finite type $X : \UU$ and a family of finite types $Y : \UU$ indexed by $X$,
the two maps $A^{(x : X) \times Y(x)} \to A$ given respectively by
$\mu_{(x : X) \times Y(x)}$ and 
\[a \mapsto \mu_X(x \mapsto \mu_Y(y \mapsto a(x,y)))\]
are identified. This expresses a kind of generalised associativity.

This is not a complete definition of $E_\infty$-monoids -- it
is missing an infinite tower of higher coherences -- but we can get far with just the data
above.

For example, we would like to know that the universe of types $\UU$ forms an
$E_\infty$-monoid where the multiplication is given by unordered join of types.
The unordered binary join would be a special case, and the Fubini map (which really should be an
equivalence) of \cref{lem:fubini} -- our main technical result -- would
be a consequence of generalised associativity, since $\mu_{X \times Y}$ and
$\mu_{Y \times X}$ are related by the path $X \times Y = Y \times X$ obtained
from univalence. In this way, the technical burden of this paper would be much
lighter if we simply had access to this $E_\infty$-monoid structure!

The problems associated with this appear at three different levels. At the first level, we
cannot even define the whole infinite tower of $E_\infty$-coherences in type
theory. This is a well-known instance of the problem of infinite objects, but is
not so relevant for us.
At the second level, it is not clear how to define the unordered join of a general
finite family of types. We have seen how to do it given a finite family of size 2,
and it is clear how to do it for 3 and 4, but the complexity grows very quickly,
and we run into the problem of infinite objects in trying to define a general pattern.
At the third level, consider what happens when we try to reason about unordered
joins of just a few spaces -- for example as in \cref{lem:fubini}. At this
level, our problems are finitary and in principle surmountable. But they are
also remarkably difficult since we lack a systematic way to approach these
problems. This story can be compared with that of coherences for the
smash product~\cite{Ljungstrom2024}.

\section{Steenrod squares and $\pi_4(\sphere{3})$}
\label{sec:applications}
\todo{more? remove completely?}
The Steenrod squares are not only an esoteric construction: they have several crucial applications in algebraic topology. One typical example is the computation of $\pi_4(\sphere{3})$, the $4$th homotopy group of the $3$-sphere, i.e.\ $\truncT{0}{\sphere{4}\to_\pt \sphere{3}}$. Although the fact that $\pi_4(\sphere{3})\cong \bZtwo$ is well known in HoTT, due to Brunerie~\cite{Brunerie16}, it was shown by Ljungström and Mörtberg~\cite{pi4-extended} that this can be shown in a very direct way under the assumption that $\pi_4(\sphere{3})$ is non-trivial. We can now complete this proof by giving a new (in HoTT) argument for why $\pi_4(\sphere{3})$ does not vanish. Let $h : \sphere{3} \to \sphere{2}$ be the \emph{Hopf map}, i.e.\ the generator of $\pi_3(\sphere{2})$~\todo{ref}. If we can show that its suspension $\Susp{h} : \sphere{4} \to \sphere{3}$ is non-trivial, we are done. In HoTT, we define $\CP \coloneq  C_h$ to be the cofibre of $h$. Since suspensions commute with cofibres, we get $C_{\Susp{h}} \simeq \Susp{\CP}$. On the other hand, the cofibre of the constant pointed map $\mathsf{const} : \sphere{4} \to \sphere{3}$ is equivalent to $\sphere{5} \vee \sphere{3}$, i.e.\ the pushout of the span $\sphere{5} \leftarrow \bOne \rightarrow \sphere{3}$. Thus, we are done if we can show the following.
\begin{theorem}
  $\Susp{\CP} \not \simeq \sphere{5} \vee \sphere{3}$
\end{theorem}
\begin{proof}
  Both spaces in questions are suspensions, with $\sphere{5} \vee \sphere{3} \simeq \Susp{(\sphere{4} \vee \sphere{2})}$. We consider the following diagram where $A\in\{\CP,\sphere{4} \vee \sphere{2}\}$ (and where, we remark, all cohomology groups are equivalent to $\bZtwo$).
  \[
\begin{tikzcd}[ampersand replacement=\&]
	{H^2(A,\bZtwo)} \& {H^4(A,\bZtwo)} \\
	{H^3(\Susp{A},\bZtwo)} \& {H^5(A,\bZtwo)}
	\arrow["{(-)^{2}}", from=1-1, to=1-2]
	\arrow["\sim"', from=1-1, to=2-1]
	\arrow[from=1-2, to=2-2]
	\arrow["{\sq{2}}"', from=2-1, to=2-2]
\end{tikzcd}
\]
This diagram is an instance of the suspension property for the Steenrod squares. When $A=\CP$, the squaring map is non-trivial (this was proved in HoTT by Brunerie~\cite{Brunerie16} using $\bZ$-coefficients but the proof works fine also for $\bZtwo$-coefficients). When $A=\sphere{4} \vee \sphere{2}$, the squaring map is trivial since wedge sums have trivial cup products. So $\Susp{\CP}$ has non-trivial $\sq{2}$ whereas for $\sphere{5} \vee \sphere{3}$ it is trivial, and thus we may conclude that these types cannot be equivalent.
\end{proof}
\begin{corollary}
  $\pi_4(\sphere{3}) \neq 0$
\end{corollary}

\section{Conclusions and Future Work}
\label{sec:conclusion}

In this paper, we have proved important properties of the Steenrod squares in
homotopy type theory, following a definition of Brunerie. This has necessitated
the development of a substantial theory of unordered HITs, and we hope that our
paper can serve as an illustration of the current state of synthetic homotopy
theory -- its power and its limitation.

This work suggests a number of further questions.
Regarding Steenrod squares, one would like to know that they generate all maps
$K(\bZtwo,n) \to K(\bZtwo, m)$ in an appropriate sense. Can this be proved in
homotopy type theory? What about the assertion that the equations listed in
\cref{thm:main} generate all possible relations between the Steenrod squares?

More generally, for any odd prime $p$ there should be an analogue of the Steenrod square,
namely the Steenrod reduced $p$th power
$P^n : K(\bZ / p \bZ,m) \to_\pt K(\bZ / p \bZ, m + 2n(p-1))$.
Can this even be studied in homotopy type theory? It is not clear how to even define it
without running into the problem of infinite objects.

Let us also highlight the problems discussed in \cref{sec:Einf}.
Is it possible to define the unordered join of a general finite family of types?
And is it possible to prove things like \cref{lem:fubini} more efficiently and systematically?

%
\bibliographystyle{IEEEtran} \bibliography{refs}

\end{document}